\documentclass[12pt]{amsart}

\usepackage{mathtools}
\usepackage{amssymb}
\usepackage{graphicx}
\usepackage{bbm}
\usepackage{float}
\usepackage{color}
\usepackage{enumerate}
\usepackage[skip=1\baselineskip plus 2pt]{parskip}
\usepackage[margin = 24mm]{geometry}
\usepackage{xurl}
\usepackage{setspace}

\newcommand{\R}{\mathbb{R}}

\newcommand{\N}{\mathbb{N}}
\newcommand{\hd}{\dim_{\textup{H}}}
\newcommand{\bd}{\dim_{\textup{B}}}
\newcommand{\ubd}{\overline{\dim}_{\textup{B}}}
\newcommand{\lbd}{\underline{\dim}_{\textup{B}}}
\newcommand{\uid}{\overline{\dim}_{\,\theta}}
\newcommand{\lid}{\underline{\dim}_{\,\theta}}

\newcommand{\ad}{\dim_{\textup{A}}}

\newcommand{\as}{\dim^\theta_{\mathrm{A}} }

\DeclarePairedDelimiter\floor{\lfloor}{\rfloor}
\newtheorem{theorem}{Theorem}[section]
\newtheorem{lemma}[theorem]{Lemma}
\newtheorem{corollary}[theorem]{Corollary}

\theoremstyle{definition}

\theoremstyle{remark}

\numberwithin{equation}{section}
\title[The Fractal structure of elliptical polynomial spirals]{The Fractal structure of elliptical polynomial spirals}

\author{S. A. Burrell}
\address{S. A. Burrell, School of Mathematics and Statistics, University of St Andrews, St Andrews, KY16 9SS, United Kingdom.}
\email{sb235@st-andrews.ac.uk}

\author{K. J. Falconer}
\address{K. J. Falconer, School of Mathematics and Statistics, University of St Andrews, St Andrews, KY16 9SS, United Kingdom.}
\email{kjf@st-andrews.ac.uk}

\author{J. M. Fraser}
\address{J. M. Fraser, School of Mathematics and Statistics, University of St Andrews, St Andrews, KY16 9SS, United Kingdom.}
\email{jmf32@st-andrews.ac.uk}

\date{August 19, 2020.}

\begin{document}
\begin{abstract}
We investigate fractal aspects of elliptical polynomial spirals; that is, planar spirals with differing polynomial rates of decay in the two axis directions. We give a full dimensional analysis of these spirals, computing explicitly their intermediate, box-counting and Assouad-type dimensions. An exciting feature is that these spirals exhibit two phase transitions within the Assouad spectrum, the first natural class of fractals known to have this property. We go on to use this dimensional information to obtain bounds for the H\"older regularity of maps that can deform one spiral into another, generalising the `winding problem’ of when spirals are bi-Lipschitz equivalent to a line segment. A novel feature is the use of fractional Brownian motion and dimension profiles to bound the H\"older exponents.

\vspace{0.1in}

\vspace{0.1in}
\end{abstract}

\maketitle

\section{Introduction}
An infinitely wound spiral is a subset of the complex plane
\begin{equation}\label{genspiral}
S(\phi) = \{\phi(t)\exp(it) : 1 < t < \infty \},
\end{equation}
where $\phi:[1, \infty) \rightarrow (0, \infty)$, known as a \emph{winding function},  is continuous, strictly decreasing and tends to zero as $t \rightarrow \infty$.
Such forms arise throughout science and the natural world, from $\alpha$-models of fluid turbulence and vortex formation to the structure of galaxies \cite{foi,mand,moff,vass,vasshunt}. The self-similarity present within these spirals makes them natural candidates for fractal analysis, and one may wish to examine the fine local structure present at the origin \cite{dup,fraser:spirals}. This may be quantified via a suitable notion of fractal dimension such as box-counting (Minkowski) dimension \cite{falconer,zub}.

The isotropic classical definition (\ref{genspiral}) may be too restrictive for the modelling of general natural or abstract phenomena. Most naturally occurring spirals are anisotropic, developing in systems with inherent asymmetry, such as elliptical whirlpools forming in a flowing body of water. Another simple example arises in Newtonian mechanics: suppose a weight attached to an elastic band is rotated about an axis parallel to the ground. At high velocities the centripetal force dominates gravity and the orbit is circular. However, if the system is allowed to decelerate, the weight will follow a spiral trajectory that will become increasingly elongated in the vertical direction as the relative contribution of gravitational force grows.

To account for these scenarios, flexibility may be introduced by controlling the rate of contraction in each axis and introducing an additional functional parameter. Thus, for two winding functions $\phi, \psi : [1, \infty) \rightarrow (0, \infty)$, we define the associated \emph{elliptical} spiral to be
\begin{equation}\label{defaffinespiral}
S(\phi, \psi) =  \{\phi(t)\cos t + i \psi(t)\sin t: 1 < t < \infty\}.
\end{equation}
Our results concern the family of elliptical \emph{polynomial} spirals $S_{p, q} = S(t^{-p}, t^{-q})$, where $0 < p \leq q$, although our arguments apply more generally. If $p=q$, then we write $S_{p, p} = S_p$ and (\ref{defaffinespiral}) recovers the \emph{generalised hyperbolic} spirals. Spirals such as these with polynomial winding functions typically arise in systems with an underlying dynamical process. On the other hand, spirals emerging from static settings are generally logarithmic with winding functions of the form $\exp(-ct)$ for $c > 0$ \cite{fraser:spirals}.

\begin{figure}[H]
\begin{center}
\includegraphics[width = \linewidth]{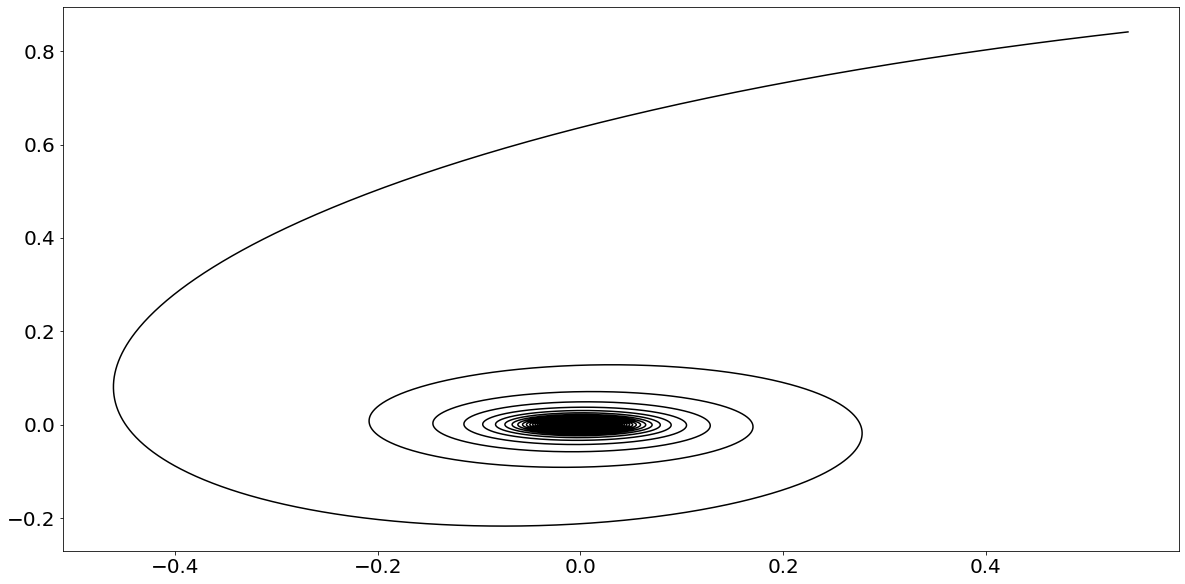}
\caption{An elliptical polynomial spiral $S_{p, q}$ with $p =0.7$ and $q=0.75$.}
\end{center}
\end{figure}

This paper serves two purposes. First, we offer a dimensional analysis of the family of elliptical polynomial spirals. This involves calculating the intermediate, box-counting (Minkowski) and Assouad-type dimensions. For a thorough introduction to these dimensions we direct the reader to \cite{falconer,jon:book}. We begin, in Theorem \ref{affineid}, by considering the intermediate dimensions of Falconer, Fraser and Kempton \cite{fafrke:2018}, which we denote $\dim_\theta$ for $\theta \in [0, 1]$ and formally define in Section \ref{intdef}. Roughly speaking, these dimensions interpolate between the Hausdorff and upper box dimensions in the sense that
$$
\hd E \leq \dim_\theta E \leq \ubd E.
$$
Intermediate dimensions have already seen surprising applications and properties, despite their recent introduction. For example, they have been used to establish relationships between the Hausdorff dimension of a set and the typical box dimension of fractional Brownian images \cite{bu:2020} or orthogonal projections \cite{bufafr:2019}. Other notable works include \cite{istvan}.

The second major notion of dimension interpolation, the Assouad spectrum of Fraser and Yu \cite{Spectraa}, lies between the upper box and Assouad dimensions and is defined in Section \ref{asdef}. One important feature of the spectrum of $S_{p, q}$ is the presence of two points of non-differentiability, or phase transitions, see Theorem \ref{affinespec}. The elliptical polynomial spirals are the first natural example to exhibit this behaviour, found before only as the product of delicate constructions.

Together, our results show the intermediate dimensions and the Assouad spectrum provide a continuous interpolation between the two extremes of the dimensional repertoire, as illustrated in Figure \ref{completeinter}.

\begin{figure}[H]
\begin{center}
\includegraphics[width = \linewidth]{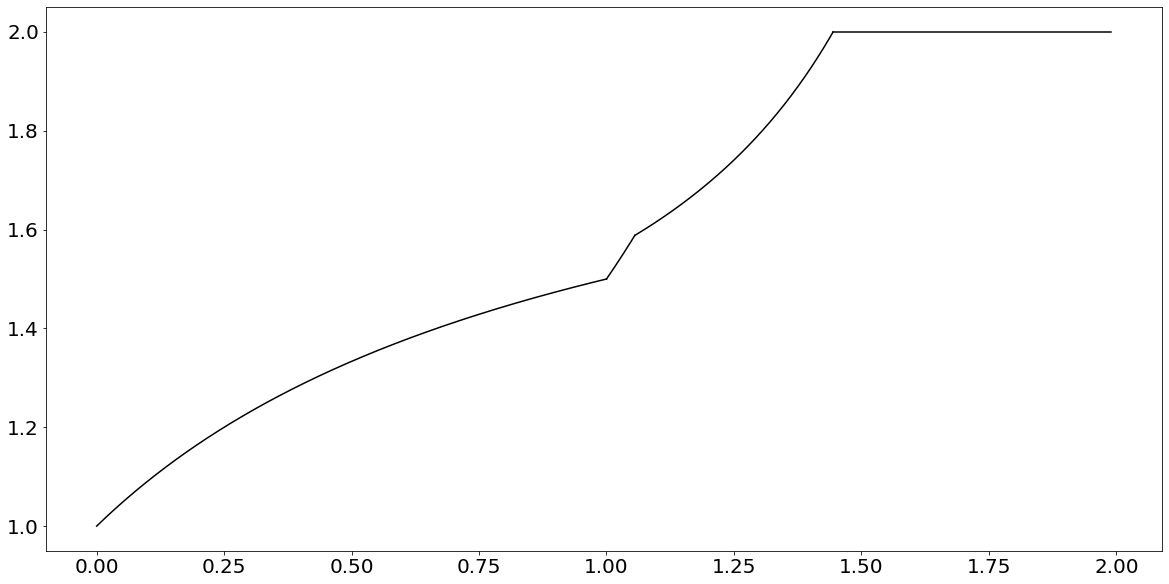}
\caption{A plot of $\dim_\theta S_{p, q}$ ($y$-axis) against $\theta$ ($x$-axis)  for $\theta \in [0, 1]$ and $\ad^{\theta-1} S_{p,q}$ against $\theta$ for $\theta \in [1,2]$. In this example, $p = 0.1$ and $q = 0.8$.}
\label{completeinter}
\end{center}
\end{figure}
The second focus is to apply the computed dimensions to determine permissible $\alpha$ such that there may exist an $\alpha$-H\"older function $f: S_{p,q} \rightarrow S_{r, s}$ that \emph{deforms} one elliptical polynomial spiral into another. Recall a function $f : X \rightarrow Y$ is $\alpha$-H\"older ($0< \alpha \leq 1$) if there exists $c > 0$ such that
$$
|f(x) - f(y)| \leq c |x-y|^\alpha \,\,\,\,\,\,\,\,\,\,(x,y \in X).
$$
Such maps may play a role within dynamical systems where spirals form and evolve over time. The H\"older exponent characterises the regularity of $f$ by quantifying the degree of distortion at local scales. A number of related questions on regularity have been explored over the past few decades for different categories of spirals that arise from winding functions of various canonical forms. Katznelson, Nag and Sullivan show that the logarithmic spiral satisfies the bi-Lipschitz \emph{winding problem} \cite{unwindspirals}. That is, it may be constructed as the image of a bi-Lipschitz homeomorphism on the unit interval. However, if $\phi$ is decays sub-exponentially, i.e.
$$
\frac{\log \phi(t)}{t} \rightarrow 0\,\,\,\,\,\,\,\,\,\,( t \rightarrow \infty),
$$
then no such bi-Lipschitz homeomorphism exists \cite{fish}. This led Fraser \cite{fraser:spirals} to investigate H\"older solutions to the winding problem for generalised hyperbolic spirals.

Our methodology is based on the dimension profiles from \cite{bu:2020,bufafr:2019}. Of course, if there is an $\alpha$-H\"older map between $S_{p, q}$ and $S_{r, s}$ we immediately obtain
\begin{equation}\label{bdbound}
\alpha \leq \frac{\dim S_{p, q}}{\dim S_{r, s}},
\end{equation}
where $\dim$ denotes Hausdorff or box-counting dimension, since
$$
\dim f(E) \leq \frac{1}{
\alpha} \dim E
$$
for $E \subset \R^n$ and $\alpha$-H\"older $f : \R^n \rightarrow \R^n$. However, the upper $2\alpha$-dimension profiles, denoted $\uid^{2\alpha} S_{p.q}$ and bounded above by $\ubd S_{p,q}$, provide a strictly sharper bound on $\alpha$ by use of the formula
\begin{equation}\label{profbound}
\alpha \leq \frac{\uid^{2\alpha} S_{p, q}}{\dim_\theta S_{r, s}},
\end{equation}
derived from Falconer \cite[Theorem 2.6]{fal:2018} in the case $\theta = 1$ and \cite[Theorem 3.1]{bu:2020} for $\theta \in [0 ,1]$.

While this approach seems promising at first sight, the definition of the profiles is potential-theoretic and rather challenging to compute in the case of $S_{p, q}$. This difficulty is circumvented by instead using the relationship to their fractional Brownian images given by Theorem \cite[Theorem 3.4]{bu:2020}. In fact, the method employed here may be used more generally to estimate the H\"older regularity of a function between any two sets for which the box or intermediate dimensions of the fractional Brownian images may be estimated from above.

\section{Statement and Discussion of results}\label{results}
This section is divided into two parts. The first offers a complete analysis of the dimensions of $S_{p, q}$, while the second considers applications to the H\"older regularity of maps that deform one elliptical polynomial spiral into another. 

\subsection{Dimensions}
For $0 < p \leq q$, the Hausdorff and packing dimensions (see \cite{falconer}) satisfy
$$
\hd S_{p, q} = \dim_{\textup{P}} S_{p, q} = 1,
$$
due to the countable stability of these dimensions and the decomposition (\ref{decomp}). We present the remaining dimensions of $S_{p, q}$ in ascending order, beginning with the intermediate dimensions.

\begin{theorem}\label{affineid}
Let $\theta \in [0, 1]$ and $0 < p \leq q$. If $p < 1$, then
$$
\dim_\theta S_{p, q} = \frac{p + q + 2\theta(1-p)}{p + q + \theta(1-p)}.
$$
Otherwise, if $p \geq 1$, then
$$
\dim_\theta S_{p, q} = 1.
$$
\end{theorem}

\begin{figure}[h!]
\begin{center}
\includegraphics[width = \linewidth]{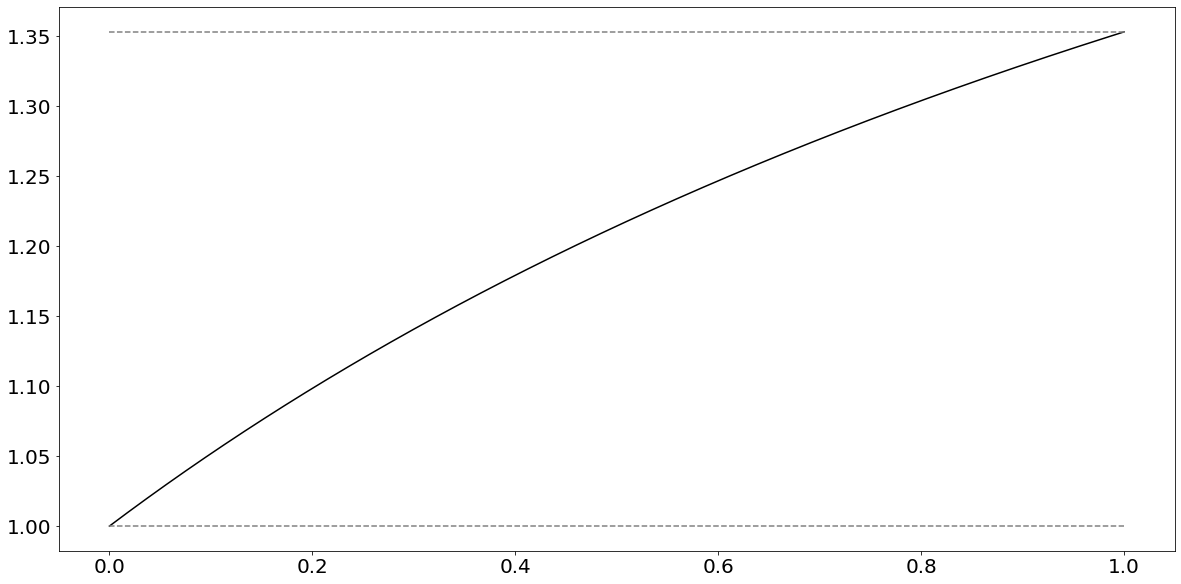}
\caption{A plot of $\dim_\theta S_{p, q}$ ($y$-axis) against $\theta$ ($x$-axis) for $p = 0.4$ and $q=0.7$, along with horizontal lines that indicate $\hd S_{p, q} = 1$ and $\bd S_{p, q} = (2+q-p)/(1+q)$.}
\end{center}
\end{figure}

In proving Theorem \ref{affineid}, it is convenient to prove the upper bound in the wider context of images of elliptical spirals under H\"older transformations. As we shall see, this becomes especially relevant in Section \ref{deformsec} when considering fractional Brownian images and dimension profiles. 

\begin{lemma}\label{holderlem}
Let $0 < p \leq q$, $\theta \in [0, 1]$ and $f : S_{p, q} \rightarrow \R^2$ be $\alpha$-H\"older $(0 < \alpha \leq 1)$. If $p < 1$, then
$$
\uid f(S_{p, q}) \leq \begin{cases}
2 & 0 < \alpha \leq 1/2\\
\frac{p + q + 2\theta(1-p)}{ \alpha(p+q) + \theta(1-p)}
& 1/2 < \alpha \leq 1
\end{cases}.
$$
Otherwise, if $p \geq 1$, then
$$
\uid f(S_{p, q}) \leq \begin{cases}
2 & 0 < \alpha \leq 1/2\\
\frac{1}{\alpha}
& 1/2 < \alpha \leq 1
\end{cases}.
$$
\end{lemma}

In Section \ref{prooflemsec}, we prove Lemma \ref{holderlem} using a direct covering argument. Theorem \ref{affineid} may then be proven by applying Lemma \ref{holderlem} to the identity map, along with a lower bound that we obtain using the mass distribution principle for intermediate dimensions \cite[Proposition 2.2]{fafrke:2018}. By setting $\theta = 1$, Theorem \ref{affineid} also offers the box dimensions of elliptical polynomial spirals.

\begin{corollary}\label{bdspsq}
Let $0 < p \leq q$. If $0 < p < 1$, then
$$
\bd S_{p, q} = \frac{2 + q - p}{1+q} = 1 + \frac{1-p}{1+q}.
$$
Otherwise, if $p \geq 1$, then
$$
\bd S_{p, q} =1.
$$
\end{corollary}

In the special case $p = q$, Theorem \ref{affineid} may be applied to determine the intermediate dimensions of generalised hyperbolic spirals, which have also been obtained independently by Tan \cite{tan}.

\begin{corollary}\label{ghpdim}
Let $\theta \in [0, 1]$. If $0 < p < 1$, then
$$
\dim_\theta S_{p} = \frac{ 2p+2\theta(1 - p)}{2p + \theta(1-p)}.
$$
Otherwise, if $p \geq 1$, then
$$
\dim_\theta S_p = 1.
$$
\end{corollary}

A question of interest within the literature on intermediate dimensions has been the classification of sets that are continuous at $\theta = 0$ \cite{bufafr:2019,fafrke:2018}. Theorem \ref{affineid} confirms that the elliptical polynomial spirals are within this class.

\begin{corollary}
Let $0 < p \leq q$. The function $\theta \rightarrow \dim_\theta S_{p, q}$ is continuous on $[0, 1]$.
\end{corollary}

Moving on into the realm of Assouad-type dimensions, Theorem \ref{affinespec} shows that these spirals exhibit two phase transitions, that is, points where the spectrum is non-differentiable. Moreover, these phase transitions are genuine in the sense that their left and right derivatives are necessarily distinct.

\begin{theorem}\label{affinespec}
Let $0 < p \leq q$. If $0 < p < 1$, then
$$
\as S_{p, q} = \begin{cases}
\frac{2+q-p}{(1+q)(1-\theta)}  &\textnormal{if $0 \leq \theta < p/(1+q)$}\\
\frac{2+q-\theta(1+q)}{(1+q)(1-\theta)}  &\textnormal{if $p/(1+q) \leq \theta < q/(1+q)$}\\
2 &\textnormal{if $q/(1+q) \leq \theta < 1$}
\end{cases}.
$$
Otherwise, if $p \geq 1$, then
$$
\as S_{p, q} = \begin{cases}
\frac{p - \theta(p-1)}{p(1-\theta)}  &\textnormal{if $0 \leq \theta < p/(1+q)$}\\
\frac{2+q-\theta(1+q)}{(1+q)(1-\theta)}  &\textnormal{if $p/(1+q) \leq \theta < q/(1+q)$}\\
2 &\textnormal{if $q/(1+q) \leq \theta < 1$}\\
\end{cases}.
$$
\end{theorem}

\begin{figure}[h!]
\begin{center}
\includegraphics[width = \linewidth]{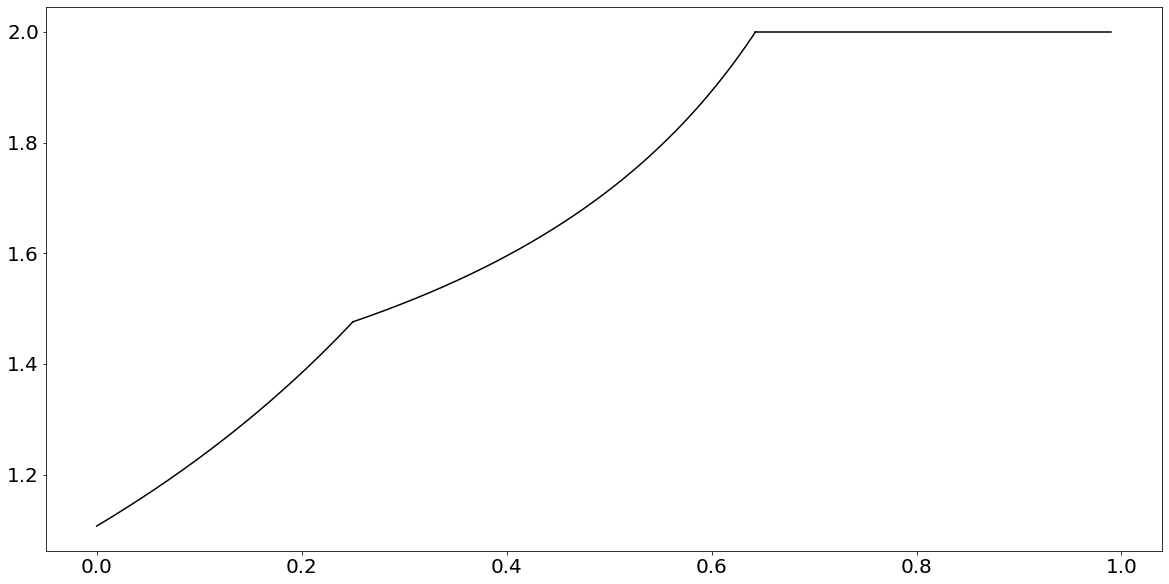}
\caption{A plot of $\as S_{p, q}$ ($y$-axis)  against $\theta$ ($x$-axis) for $p = 1.1$ and $q = 1.8$.}
\end{center}
\end{figure}
The reader familiar with \cite{fraser:spirals} may be surprised to see that the first phase transition occurs at $p/(1+q)$, rather than $p/(1+p)$. Indeed, this shows an unexpected and subtle interaction between the parameters. Theorem \ref{affinespec} also shows that elliptical polynomial spirals have maximal Assouad dimension.
\begin{corollary}
For all $0 < p \leq q$, $\ad S_{p, q} = 2$.
\end{corollary}

Lastly, the relationship between elliptical polynomial spirals and concentric ellipses is worthy of comment. Let us define
$$
C_{p, q} = \bigcup\limits_{n \in \N} E((2\pi n)^{-p}, (2\pi n)^{-q})
$$
where $E(x, y)$ ($x \geq y$) denotes the ellipse centred on the origin with major axis of length $2x$ and minor axis of length $2y$. See Figure \ref{conelip}. It is not surprising that $C_{p, q}$ is dimensionally equivalent to $S_{p, q}$ and our arguments apply equally well to such sets, since it is not too hard to show that the covering number of $S_{p, q}^k$ is equal to that of $E((2\pi k)^{-p}, (2\pi k)^{-q})$ up to multiplicative constants depending only on $p$ and $q$.

\begin{corollary}
Theorem \ref{affineid} and Theorem \ref{affinespec} hold with $S_{p, q}$ replaced by $C_{p, q}$.
\begin{proof}
This  follows immediately upon observing that $S_{p, q} \cap \{ z \in \mathbb{C} : \textnormal{Re}(z) < 0\}$ is bi-Lipschitz equivalent to $C_{p, q} \cap \{ z \in \mathbb{C} : \textnormal{Re}(z) < 0\}$.
\end{proof}
\end{corollary}

\begin{figure}[h!]
\begin{center}
\includegraphics[width = \linewidth]{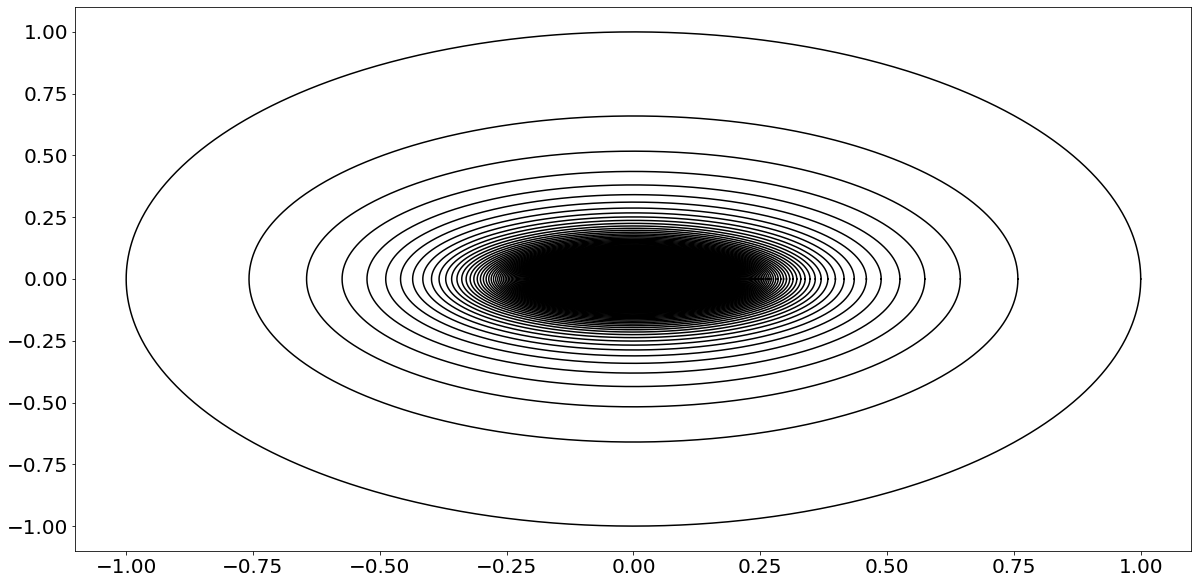}
\caption{A family of concentric ellipses $C_{p, q}$ dimensionally equivalent to $S_{p, q}$, where $p = 0.4$ and $q = 0.6$.}
\label{conelip}
\end{center}

\end{figure}

\subsection{Applications}\label{deformsec}

In this section we use dimension theoretic information to examine the regularity of H\"older mappings that deform one elliptical polynomial spiral into another. The behaviour of dimension under H\"older mappings has been widely studied, and offers insight into permissible $\alpha$ for which there may exist an $\alpha$-H\"older map transforming a set $X$ onto a set $Y$. For example, Corollary \ref{bdspsq} allows us to glean such information from the box dimensions of $S_{p, q}$ and $S_{r, s}$.

\begin{theorem}\label{bdregthm}
Let $0 < p \leq q$ and $0 < r \leq s$ with $r \leq 1$. Suppose $f : S_{p, q} \rightarrow S_{r, s}$ is $\alpha$-H\"older. If $p \leq 1$, then
$$
\alpha \leq \frac{(2+q-p)(1+s)}{(2+s-r)(1+q)}.
$$
Otherwise, if $p > 1$, then
$$
\alpha \leq \frac{1 + s}{2 + s - r}.
$$
\end{theorem}
\begin{proof}
Let $p \leq 1$. By the standard properties of box-counting dimensions, see \cite[Chapter 2]{falconer},
$$
\frac{2 + s - r}{1+s} = \bd f(S_{p, q}) \leq \frac{1}{\alpha} \bd S_{p, q} = \frac{1}{\alpha}\frac{2 + q - p}{1+q},
$$
from which the first result follows. The case for $p > 1$ is similar.
\end{proof}

Theorem \ref{bdregthm} provides a non-trivial bound on $\alpha$ when $\bd S_{r, s} > \bd S_{p, q}$. However, it is possible to do better using dimension profiles. Intuitively, the $m$-dimensional profile may be thought of as the dimension of an object when viewed from an $m$-dimensional viewpoint. In favour of brevity we omit a thorough introduction to dimension profiles, which may be found in \cite{bufafr:2019}. In the following lemma, we bound the upper $2\alpha$-profiles of $S_{p, q}$, denoted $\uid^{2\alpha} S_{p, q}$, by a quantity strictly less than the dimension for $\theta > 0$, $p < 1$ and $1/2 < \alpha < 1$. This is depicted in Figure \ref{profcompfig}. 

\begin{lemma}\label{profilelemma}
Let $0 < p \leq q$ and $\theta \in [0, 1]$. If $p \leq 1$, then
\begin{equation*}\label{profilebound}
\uid^{2\alpha} S_{p, q} \leq \begin{cases}
2\alpha & 0 < \alpha \leq 1/2\\
\frac{\alpha(p + q + 2\theta(1-p))}{\alpha(p+q)+\theta(1-p) }
& 1/2 < \alpha < 1
\end{cases}.
\end{equation*}
\begin{proof}
Index-$\alpha$ fractional Brownian motion is almost surely $(\alpha -\varepsilon)$-H\"older for all $\varepsilon > 0$ \cite{kahane:book}. Hence, for each $\varepsilon > 0$, Lemma \ref{holderlem} tells us that
$$
\uid B_\alpha(S_{p, q}) \leq \begin{cases}
2 & 0 < \alpha \leq 1/2\\
\frac{p + q + 2\theta(1-p)}{ (\alpha -\varepsilon)(p+q) + \theta(1-p)}
& 1/2 < \alpha < 1
\end{cases}
$$
almost surely. Then, letting $\varepsilon \rightarrow 0$, by \cite[Theorem 3.4]{bu:2020} we have
$$
\uid^{2\alpha} S_{p, q} = \alpha \uid B_\alpha(S_{p, q})\leq \begin{cases}
2\alpha & 0 < \alpha \leq 1/2\\
\frac{\alpha(p + q + 2\theta(1-p))}{\alpha(p+q)+\theta(1-p) }
& 1/2 < \alpha < 1
\end{cases}
$$
almost surely. This concludes the proof, since $\uid^{2\alpha} S_{p, q}$ has no random component.
\end{proof}
\end{lemma}

\begin{figure}[h!]
\begin{center}
\includegraphics[width = \linewidth]{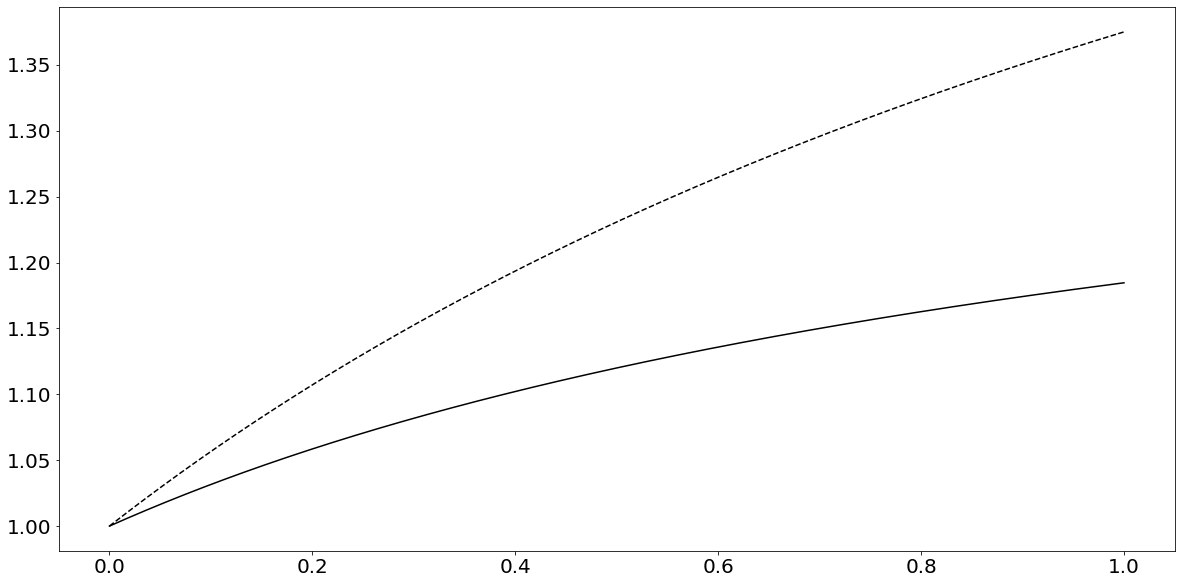}
\caption{A plot of $\dim_\theta S_{p,q}$ (dashed) and the upper bound of $\uid^{2\alpha} S_{p, q}$ ($y$-axis) against $\theta$ ($x$-axis) for $\alpha = 0.7$, $p = 0.4$ and $q = 0.6$.}
\label{profcompfig}
\end{center}
\end{figure}

It is clear from Lemma \ref{profilelemma} that we may produce a bound strictly superior to that from Theorem \ref{bdregthm} for all parameter configurations with $p < 1$ using dimension profiles. This improvement is illustrated in Figure \ref{alphaestfig}. For larger $p$, the two approaches are equivalent.

\begin{figure}[ht]
\begin{center}
\includegraphics[width = \linewidth]{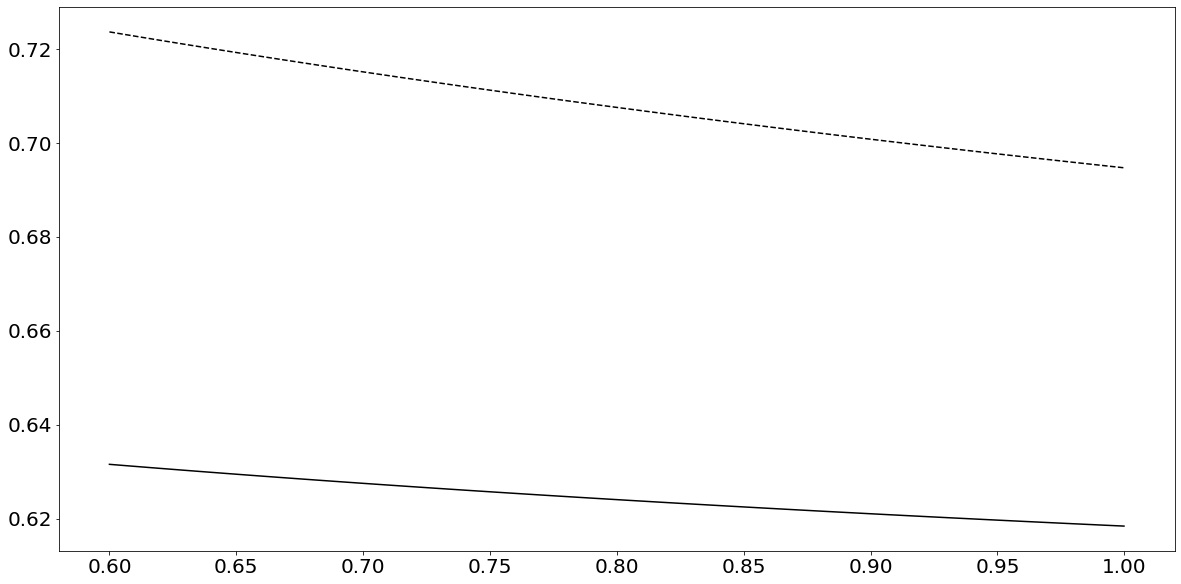}
\caption{Bounds on the H\"older exponent of $f : S_{p,q} \rightarrow S_{r, s}$ against the value of $q$ (x-axis) when $p = 0.6, r=0.2$ and $s= 0.1$. The bounds derived from the dimension profiles (Theorem \ref{main}) and the box-counting dimension (Theorem \ref{bdregthm}) correspond to the solid and dashed lines, respectively.}\label{alphaestfig}
\end{center}
\end{figure}

\begin{theorem}\label{main}
Let $0 < p \leq q$ and $0 < r \leq s$. If $p \leq 1$, $r \leq 1$ and
$f : S_{p, q} \rightarrow S_{r, s}$ is $\alpha$-H\"older, then
$$
\alpha \leq \frac{p + q + r + s - pr + qs}{(2 + s - r)(p + q)}.
$$
\begin{proof}
The target bound is strictly greater than $1/2$, and so we may assume without loss of generality that $\alpha > 1/2$. The discrepancy between the profile and the dimension is maximised when $\theta = 1$. Thus, set $\theta = 1$, and observe from (\ref{profbound}), Lemma \ref{profilelemma} and Corollary \ref{bdspsq} that
$$
\dim_1 S_{r, s} =  \frac{2 + s -r}{1+s}\leq \frac{1}{\alpha} \overline{\dim}_1^{2\alpha} S_{p, q}\leq
 \frac{p + q + 2(1-p)}{\alpha(p+q)+(1-p)},
$$
from which the result follows on re-expressing the inequality in terms of $\alpha$. 
\end{proof}
\end{theorem}

Recall that if $p=q$, then $S_{p, p} = S_p$ is a generalised hyperbolic spiral. In this case, Theorem \ref{main} offers an appealing upper bound on $\alpha$.

\begin{corollary}\label{spsqthm}
Let $p > q$ and $f : S_p \rightarrow S_q$ be $\alpha$-H\"older. If $p \leq 1$, then
$$
\alpha \leq \frac{p + q}{2p}.
$$
\end{corollary}
\begin{proof}
Apply Theorem \ref{main} to $f : S_{p, p} \rightarrow S_{q, q}$. 
\end{proof}

In \cite{fraser:spirals}, it was seen that the Assouad spectrum provided the most information on H\"older exponents in the context of the winding problem (mapping a line segment to a spiral). However, it is easily verified that the same tool, \cite[Theorem 4.11]{Spectraa}, provides only trivial information in our setting (mapping a spiral to a spiral). Conversely, in the context of the winding problem, dimension profiles provide no new information. Thus, it is interesting to see that the regimes are inverted in the context of spiral deformation, with the Assouad spectrum providing the least information and the dimension profiles the most.

\section{Preliminaries}\label{prelims}

In preparation for the main proofs, we begin this subsection by setting notation and making a few technical geometric observations. Afterwards, in order to serve as a reference point, we formally define a selection of the dimension theoretic concepts. However, we assume basic familiarity with topics such as Hausdorff dimension and measure, and direct the reader to the classic text \cite{falconer} for a thorough exposition on the fundamentals of dimension theory.

\subsection{Decomposition, notation, and geometric observations}Dimension concerns limiting processes for which fixed multiplicative constants are typically of little consequence. Therefore, we often write $x \lesssim y$ when it is clear there exists a uniform constant $c > 0$ not depending on $x$ and $y$ such that $x \leq c y$. Naturally, we analogously define $\gtrsim$, and write $x \approx y$ if $x \lesssim y$ and $x \gtrsim y$. In circumstances where $c$ is not uniform but depends on certain parameters, say $t_1$, $t_2, \dots$, we write $\lesssim_{t_1,t_2,\dots}$, $\gtrsim_{t_1,t_2,\dots}$ and $\approx_{t_1,t_2,\dots}$ to make this clear.

A useful trick is to decompose $S_{p, q}$ into a countable disjoint union of \emph{full turns}. In particular, we define
\begin{equation}\label{decomp}
S_{p, q} := \bigcup\limits_{k \geq 1} S_{p, q}^k,
\end{equation}
where
$$
S_{p, q}^k = \{t^{-p}\cos t + i t^{-q}\sin t : 2\pi k \leq t < 2\pi(k+1)\}.
$$
Note that, for arithmetic convenience, we have removed the part of $S_{p, q}$ corresponding to $1 < t < 2\pi$ in the definition (\ref{defaffinespiral}) without meaningful loss of generality. The following geometric observation estimates the sum of the $1$-dimensional Hausdorff measures, or length, over a collection of consecutive turns using standard number theoretic estimates.

\begin{lemma}\label{circumlem}
Let $ 0 < p \leq q$. For $k \geq 1$, 
\begin{equation}\label{circim}
\mathcal{H}^1(S_{p, q}^k) \approx_p k^{-p}
\end{equation}
Moreover, for sufficiently large integers $N, M \in \N$ with $M < N$,
\begin{equation}\label{sumest}
\sum\limits_{k = M}^{N} \mathcal{H}^1(S_{p, q}^k) \approx_p
\begin{cases}
N^{1-p} - M^{1-p} &\textnormal{if $p < 1$}\\
\log N - \log M &\textnormal{if $p = 1$}\\
M^{1-p} - N^{1-p} &\textnormal{if $p > 1$}
\end{cases}.
\end{equation}
\end{lemma}

\begin{proof}
By comparing $\mathcal{H}^1(S_{p, q}^k)$ with the perimeter of a square of sidelength $2(2k\pi)^{-p}$ centred on the origin we may deduce
$$
(2k\pi)^{-p} \leq \mathcal{H}^1(S_{p, q}^k)  \leq 8 (2k\pi)^{-p},
$$
from which (\ref{circim}) follows immediately. (\ref{sumest}) may then be deduced in a standard way. Letting $\floor{t}$ denote the integer part of $t \in \R$, observe that for $p \neq 1$,
\begin{align*}
\sum\limits_{k = M}^{N} \mathcal{H}^1(S_{p, q}^k)  \approx_p \sum\limits_{k = M}^{N} k^{-p}  = \sum\limits_{k = M}^{N} \,\,\int\limits_{k}^{k + 1} \floor{u}^{-p} \,du  \approx_p \frac{1}{1-p} (N^{1-p} - M^{1-p}).
\end{align*}
The case for $p = 1$ follows similarly.
\end{proof}

\subsection{Intermediate dimensions}\label{intdef}
The intermediate dimensions are a family of dimensions, indexed by $\theta \in [0, 1]$ and introduced in \cite{fafrke:2018}, that interpolate between the Hausdorff and upper box counting dimensions.

For bounded $E \subset \R^n$ and $0 < \theta \leq 1$, the  {\em lower intermediate dimension} of $E$ may be defined as
\begin{align*}
\lid E =  \inf \big\{& s\geq 0  :  \mbox{ \rm for all $\epsilon >0$ and all $\delta_0>0$, there exists }  \nonumber\\
&\mbox{$0<\delta\leq \delta_0$ and a cover $ \{U_i\} $ of $E$  such that} \\
 & \mbox{$\delta\leq  |U_i| \leq \delta^\theta $ and 
 $\sum |U_i|^s \leq \epsilon$}  \big\}\nonumber
\end{align*}
and the corresponding {\em upper intermediate dimension} by
\begin{align*}
\uid E =  \inf \big\{& s\geq 0  :  \mbox{ \rm for all $\epsilon >0$, there exists $\delta_0>0$ such that} \nonumber\\
& \mbox{for all $0<\delta\leq \delta_0$, there is a cover $ \{U_i\} $ of $E$} \\
&\mbox{such that $\delta \leq  |U_i| \leq \delta^\theta$ and 
$\sum |U_i|^s \leq \epsilon$}  \big\},\nonumber
\end{align*}
where $|U|$ denotes the diameter of a set $U \subset \R^n$. For $\theta = 0$, define $$\underline{\dim}_0 E = \overline{\dim}_0 E = \hd E,$$
while at $\theta = 1$ it is clear that
$$\lbd E = \underline{\dim}_1 E \textnormal{\,\,\,\,\,\,and\,\,\,\,\,\,}\ubd E = \overline{\dim}_1 E.$$
If $\lid E = \uid E$ we say the $\theta$-intermediate dimension of $E$ \emph{exists} and write $\dim_\theta E$.

\subsection{The Assouad spectrum and dimensions}\label{asdef}

The Assouad spectrum of $F$, a family of dimensions indexed by $\theta \in [0, 1)$ and introduced in \cite{Spectraa}, interpolates between the upper box dimension and the quasi-Assouad dimension. Formally, it is the function $\theta \mapsto \as F$ defined by
\begin{align*}
\as F =  \inf \big\{& \alpha \geq 0  :  \mbox{ \text{$\exists$ $C >0$ such that, for all $0<r<1$ and $x \in F$,} } \nonumber\\
&\mbox{\text{$ N_{r} \big( B(x,r^\theta) \cap F \big) \ \leq \ C(r^\theta / r)^\alpha$ }}  \big\},\nonumber
\end{align*}
where $N_r(E)$ denotes the smallest number of hypercubes of sidelength $r$ required to cover $E$. The Assouad dimension is defined similarly but considers $N_r(B(x, R) \cap F)$ for arbitrary $0 < r < R$, thus removing the restriction on the precise relationship imposed by $\theta$. The limit as $\theta \rightarrow 1$ is known as the quasi-Assouad dimension and, as we shall see, in the context of spirals is equal to the Assouad dimension. For a detailed treatment of Assouad-type dimensions and their various applications we direct the reader to \cite{jon:book}.

\section{Proofs}
\subsection{Proof of Lemma \ref{holderlem}}\label{prooflemsec}
\begin{proof}[\unskip\nopunct] 
Let $0 \leq s \leq 2$ and $0 < \delta < 1$. To aid readability when dealing with particularly complicated exponents, we write $t= -\log \delta$. 

If $0 < \alpha \leq 1/2$, the bound is trivial. Thus, hereafter assume $1/2 < \alpha \leq 1$.

Choose $M \in \N$ to be the smallest integer satisfying
\begin{equation}\label{Mlab}
M \geq \exp\left(\frac{t(s - (1/\alpha) + \theta(2-s))}{1-p + \alpha(p+q)}\right),
\end{equation}
and note that by (\ref{circim}) from Lemma \ref{circumlem},

\begin{equation}\label{precover}
N_{\delta^{1/\alpha}}(S_{p,q}^k) \approx_p \frac{k^{-p}}{\delta^{1/\alpha}}.
\end{equation}
Let the uniform constant associated with the H\"older property of $f$ be $c > 0$. Then, for $k \leq M$, by considering the image of a cover satisfying (\ref{precover}) under $f$, we may obtain a cover of $f(S_{p, q}^k)$ by at most
$$
\approx_{p} \frac{k^{-p}}{\delta^{1/\alpha}}
$$
balls of diameter $c2^{\alpha/2}\delta$. It follows that there exists a constant $d_{c, p, \alpha}$, depending only on $c$, $p$ and $\alpha$, such that we may cover $f(S_{p, q}^k)$ by
$$
d_{c, p, \alpha} \frac{k^{-p}}{\delta^{1/\alpha}} \approx_{c, p, \alpha} \frac{k^{-p}}{\delta^{1/\alpha}}
$$
balls of diameter $\delta$. The remaining region will be covered by balls of diameter $\delta^\theta$. For $k > M$,
\begin{align*}
\bigcup\limits_{k > M} f(S_{p, q}^k) &\subset f( [-M^{-p}, M^{-p}] \times [-M^{-q}, M^{-q}]) \\
&\subseteq [-cM^{-p\alpha}, cM^{-p\alpha}] \times [-cM^{-q\alpha}, cM^{-q\alpha}],
\end{align*}
and such a rectangle may be covered by 
$$
\approx_c \frac{M^{-(p + q)\alpha}}{\delta^{2\theta}}
$$
balls of diameter $\delta^\theta$. Summing over this cover, that we denote $\{U_i\}_i$, gives
\begin{align}\label{coversum}
\sum\limits |U_i|^s \approx_{c, p, \alpha}  \left(\frac{M^{-\alpha(p + q)}}{\delta^{2\theta}}\right) \delta^{\theta s} + \delta^{s}\sum\limits_{k = 1}^{M} \frac{k^{-p}}{\delta^{1/\alpha}} .
\end{align}
If $p \leq 1$, then (\ref{Mlab}) and (\ref{coversum}) imply
\begin{align}\label{prevcalc}
\sum\limits |U_i|^s &\approx_{c, p, \alpha}  M^{-\alpha(p + q)} \delta^{\theta s - 2\theta} + M^{1-p}\delta^{s-(1/\alpha)}\nonumber\\
&\approx_{c, p, \alpha} 2\exp\left(-t\frac{s(\alpha(p+q) + \theta(1-p) ) - (p + q + 2\theta(1-p))}{1 - p + \alpha(p + q)}\right).
\end{align}
Hence, $ \sum |U_i|^s \rightarrow 0$ as $\delta \rightarrow 0$ providing
$$
s > \frac{p + q + 2\theta(1-p)}{\alpha(p+ q) + \theta(1-p)},
$$
and so
$$
\overline{\dim}_\theta f(S_{p, q}) \leq \frac{p + q+ 2\theta(1-p)}{\alpha(p + q) + \theta(1-p)}.
$$
Note that if $p = 1$ this bound equals $1/\alpha$, as required. On the other hand, if $p > 1$, then (\ref{coversum}) implies
\begin{align*}
\sum\limits |U_i|^s &\approx_{c, p, \alpha}  M^{-\alpha(p + q)} \delta^{\theta s - 2\theta} + \delta^{s-(1/\alpha)}\\
&\approx_{c, p, \alpha}  \exp\left(-t\frac{s(\alpha(p+q) + \theta(1-p) ) - (p + q + 2\theta(1-p))}{1 - p + \alpha(p + q)}\right) + \delta^{s- (1/\alpha)}.
\end{align*}
Clearly,
$$
1-p + \alpha(p+q) \geq 1-p + \frac{1}{2}(p + p) = 1,
$$
and so the left-hand term converges to $0$ as $\delta \rightarrow 0$ if 
$$
s > \frac{p + q + 2\theta(1-p)}{\alpha(p + q) + \theta(1 - p)},
$$
while the right hand term requires $s > 1/\alpha$. Hence
$$
 \overline{\dim}_\theta f(S_{p, q}) \leq \max\left\{\frac{p + q + 2\theta(1-p)}{\alpha(p + q) + \theta(1 - p)}, \frac{1}{\alpha}\right\} = \frac{1}{\alpha}.
$$
\end{proof}
\subsection{Proof of Theorem \ref{affineid}}
\begin{proof}[\unskip\nopunct]
The upper bound follows from Lemma \ref{holderlem} applied to the identity mapping. If $p \geq 1$, the upper bound coincides with the trivial lower bound, and so it suffices to assume $0 < p < 1$. Let $0 < \delta < 1$, and define $M \in \N$ to be the smallest integer satisfying
\begin{equation*}
M \geq \exp\left(\frac{t(s - 1 + \theta(2-s))}{1 + q}\right),
\end{equation*}
recalling $t= -\log \delta$. Next, define
$$
s = \frac{p + q + 2\theta(1-p)}{p + q + \theta(1-p)},
$$
and construct a measure $\mu_\delta$ supported on $S_{p, q}^+$ by
\begin{equation}\label{spiralmeasure}
\mu_\delta = \delta^{s-1} \sum\limits_{k = 1}^{M} \mathcal{H}^1 \big|_{S_{p,q}^{+, k}},
\end{equation}
where $\mathcal{H}^1 \big|_{S_{p,q}^{+, k}}$ denotes the restriction of $1$-dimensional Hausdorff measure to $S_{p,q}^{+, k}$.

It is easy to see that
$$
\mu_\delta(S_{p,q}^+) = \delta^{s-1} \sum\limits_{k = 1}^{M} \mathcal{H}^1 (S_{p,q}^{+, k}) \gtrsim_p \delta^{s-1} \sum\limits_{k =1}^{M} k^{-p} \approx_p M^{1-p}\delta^{s-1} \approx_p 1,
$$
with the final calculation similar to that which obtained (\ref{prevcalc}).

Next, in order to apply the mass distribution principle for intermediate dimensions, we must estimate $\mu_\delta(U)$ for arbitrary Borel sets $U$ satisfying $\delta \leq |U| \leq \delta^{\theta}$. First, observe that
$$
\left(\frac{1}{(k-1)^q} - \frac{1}{k^q}\right)-\left(\frac{1}{(k-1)^p} - \frac{1}{k^p}\right) = \frac{1 - (k-1)^{q-p}}{(k-1)^q} - \frac{k^{q-p} - 1}{k^q} \leq 0
$$
for $k > 1$, since $p \leq q$. Hence, up to multiplicative constants depending only on $p$ and $q$, consecutive turns of the spiral are separated by at least
$$
\frac{1}{(k-1)^q} - \frac{1}{k^q}.
$$
An application of the mean value theorem then gives
$$
\frac{1}{(k-1)^q} - \frac{1}{k^q} \geq \frac{q}{k^{q+1}} \geq \frac{q}{M^{1 + q}}
$$
for $2 \leq k \leq M$. It follows that a set $U$ satisfying $\delta \leq |U| \leq \delta^{\theta}$ may intersect at most $|U|M^{1+ q}$ turns that contain mass, up to a constant depending only on $p$ and $q$. Moreover, for each turn it intersects, $U$ may cover a region of mass at most $\delta^{s-1}$ multiplied by the circumference of a ball of diameter $U$. Hence
\begin{align*}
\mu_\delta(U) &\lesssim_{p, q}  (|U|\delta^{s-1})(|U|M^{1 + q}) \\
&= |U|^2 \delta^{s-1}\delta^{-s + 1 - \theta(2-s)}\\
&= |U|^2 \delta^{\theta(s- 2)}\\
&\leq |U|^2 |U|^{s-2} \textnormal{  (since $s < 2$ and $|U| \leq \delta^\theta$)}\\
&= |U|^s.
\end{align*}
The lower bound then follows from the mass distribution principle for intermediate dimensions, see \cite[Proposition 2.2]{fafrke:2018}.
\end{proof}

It is worth remarking that measures of a form similar to (\ref{spiralmeasure}) could be useful for a wide range of sets $E$ with a spiral structure. For example, we might consider the image of a spiral under a map $f$ that distorts the local geometry while preserving the general form. If it were the case that $\hd f(S_{p,q}^k) = t$ for all $k \in \N$, then measures of the form
\begin{equation}
\mu_\delta = \delta^{s-t} \sum\limits_{k = 1}^{M} \mathcal{H}^t \big|_{f(S_{p,q}^k)}
\end{equation}
may be good candidates for use with \cite[Proposition 2.2]{fafrke:2018}.

\subsection{Proof of Theorem \ref{affinespec}}
\begin{proof}[\unskip\nopunct]
If $p = q$, then the result is \cite[Theorem 4.4]{fraser:spirals}, so let $0 < p < q$. For each $0 < \delta < 1$, define $L_p, L_q \in \N$ to be the largest integers such that
\begin{equation}\label{Lpdef}
 \delta \leq \frac{1}{(\pi + 2\pi L_p)^{p}} - \frac{1}{(\pi + 2\pi (L_p + 1) )^{p}}
\end{equation}
and
\begin{equation}\label{Lqdef}
 \delta \leq \frac{1}{(\frac{3\pi}{2} + 2\pi L_q)^{q}} - \frac{1}{(\frac{3\pi}{2} + 2\pi (L_q + 1) )^{q}}.
\end{equation}
Geometrically, $L_p$ and $L_q$ are the maximal indices $k$, such that $S_{p, q}^k$ is separated on the horizontal and vertical axes by at least $\delta$, respectively.  In addition, define the integers $l_p$ and $l_q$ to be the minimal $k$ such that $S_{p, q}^k$ intersects the ball $B(0, \delta^\theta)$ on the horizontal and vertical axes, respectively. In particular,
$$\left(\pi + 2 \pi l_p\right)^{-p} \leq \delta^\theta < \left(\pi + 2\pi (l_p-1)\right)^{-p}$$ 
and
$$\left(\frac{3\pi}{2} + 2 \pi l_q\right)^{-q} \leq \delta^\theta < \left(\frac{3\pi}{2} + 2\pi (l_q-1)\right)^{-q}.$$ 

Throughout, we use the fact that
$$
S_{p, q} \cap B(0, \delta^\theta) \subseteq \bigcup\limits_{k=l_q}^{\infty} S_{p, q}^k \cap B(0, \delta^\theta).
$$


The ordering of $L_p, L_q, l_p$ and $l_q$ depends on $\theta$, and gives rise to phase transitions within the spectrum. To determine the order based on a value of $\theta$, first note that 
\begin{equation}\label{ltapprox}
l_t \approx_t \delta^{-{\theta}/{t}}
\end{equation}  
for $t \in \{p, q\}$. Then, for $t  \in \{p, q\}$, it follows from an application of the mean value theorem applied to $f(x) = x^{-t}$ that
$$
\frac{t}{(k+1)^{1+t}} \leq \frac{1}{k^t} - \frac{1}{(k+1)^t} \leq \frac{t}{k^{1+t}}.
$$
This, along with the fact $L_p$ and $L_q$ are the maximal integers satisfying (\ref{Lpdef}) and (\ref{Lqdef}), respectively, implies
\begin{equation}\label{Lpapprox}
L_t \approx_t \delta^{-\frac{1}{1+t}}.
\end{equation}
It is immediate that $l_p \gtrsim_{p, q} l_q$ and $L_p \gtrsim_{p,q} L_q$ for all $\theta \in [0, 1)$ since $p < q$, but we must divide into cases to learn more. By continuity of the Assouad spectrum \cite[Corollary 3.5]{Spectraa} and \cite[Corollary 3.6]{Spectraa}, it suffices to consider $\theta$ in the ranges $0 \leq \theta < p/(1 + q)$ and $p/(1 + q) < \theta < q/(1 + q)$. Throughout, we use the estimate
\begin{equation}\label{rev}
N_\delta(S_{p,q} \cap B(z, \delta^\theta)) \lesssim_{p, q} N_\delta(S_{p, q} \cap B(0, \delta^\theta))
\end{equation}
for all $z \in \mathbb{C}$. This reduction in intuitively clear, since the origin is the densest part of the set $S_{p, q}$ and can be shown via a similar argument to \cite[Theorem 4.4]{fraser:spirals}, which covers the case $p = q$. In particular, if $|z| < 2\delta^\theta$, then subsequent arguments with $B(0, \delta^\theta$) are easily modified up to uniform constants since $B(z,\delta^\theta) \subseteq B(0, 3\delta^\theta)$. On the other hand, if $|z| \geq 2\delta^\theta$ and
$
B(z, \delta^\theta) \cap S_{p, q}^k \neq \emptyset
$
for some $k \geq 1$, then $k^{-p} \gtrsim \delta^\theta$ or $k^{-q} \gtrsim \delta^\theta$, recalling the intersections of $S_{p, q}^k$ with the horizontal and vertical axes are (up to constants) $k^{-p}$ and $k^{-q}$, respectively. Since $p \leq q$, both conditions hold if $k \lesssim \delta^{-\theta/p}$ and $\delta < 1$. Summing over permissible $k \geq 1$ implies
$$
N_\delta(B(z, \delta^\theta) \cap S_{p, q}) \lesssim_{p, q}  \left(\delta^{-\frac{\theta}{p}}\right) \frac{\delta^\theta}{\delta} = \left(\frac{\delta^{\theta}}{\delta}\right)^{\frac{p - (p-1)\theta}{(1-\theta)p}}
$$
as in \cite{fraser:spirals}. This is sufficient to prove \eqref{rev}, since the below proofs show
$$
N_\delta(B(0, \delta^\theta) \cap S_{p, q}) \gtrsim_{p,q} \left(\frac{\delta^{\theta}}{\delta}\right)^{\frac{p - (p-1)\theta}{(1-\theta)p}}
$$
in all cases.

\subsubsection*{\textnormal{\textbf{Case 1:} suppose $\frac{p}{1+q} < \theta < \frac{q}{1+q}$.}}
In order to simplify some geometric estimates, it is convenient to adopt an equivalent definition of the Assouad spectrum in this case. Specifically, we consider minimal coverings of the set $D(0, \delta^\theta) \cap S_{p, q}$, where $D(0, \delta^\theta)$ is a square centred on the origin of sidelength $2\delta^\theta$ and orientated with the co-ordinate axes. By (\ref{ltapprox}) and (\ref{Lpapprox}), for sufficiently small $\delta > 0$,
$$
l_p^{-p} < L_q^{-p} < l_q^{-p}.
$$
For $l_q \leq k \leq L_q$, the set $S_{p, q}^k \cap D(0, \delta^\theta)$ contains at least one arc $A$ such that 
$$
\mathcal{H}^1(A) \approx \delta^\theta,
$$
and so
$$
N_\delta(A) \approx \frac{\delta^\theta}{\delta}.
$$
Turns in the range $l_q \leq k \leq L_q$ are separated by at least $\delta$ on the vertical and horizontal axes, and thus any square of sidelength $\delta$ may intersect at most two of the corresponding arcs. 

It follows that, recalling (\ref{ltapprox}) and (\ref{Lpapprox}),
\begin{align}\label{case2lb}
N_\delta(S_{p, q} \cap D(0, \delta^\theta)) &\gtrsim \sum\limits_{k=l_q}^{L_q}  \delta^{\theta - 1} \\
&\approx_{p, q} \delta^{\theta - 1}\left(\delta^{-\frac{1}{1+q}} - \delta^{-\frac{\theta}{q}}\right) \nonumber\\
&\gtrsim_{p, q} \left(\frac{\delta^\theta}{\delta}\right)^{\frac{2+q-\theta(1+q)}{(1+q)(1-\theta)}}.\nonumber
\end{align}
Hence
$$\as S_{p, q} \geq \frac{2+q-\theta(1+q)}{(1+q)(1-\theta)}.
$$
On the other hand, observe
$$
\bigcup\limits_{k = L_q}^{\infty} S_{p, q}^k \cap D(0, \delta^\theta) \subseteq [-\delta^\theta, \delta^\theta] \times [-(2\pi L_q)^{-q}, (2\pi L_q)^{-q}],
$$
and such a rectangle may be covered by $$\approx_q \frac{\delta^\theta L_q^{-q}}{\delta^2}$$
squares of sidelength $\delta$. The remaining portion may be covered in a similar manner as in (\ref{case2lb}), and we conclude
\begin{align*}
N_\delta(S_{p, q} \cap B(0, \delta^\theta)) &\lesssim_q \frac{\delta^\theta L_q^{-q}}{\delta^2} + \sum\limits_{k = l_q}^{L_q} \delta^{\theta - 1}\\
&\approx_{p, q} \left(\frac{\delta^\theta}{\delta}\right)^{\frac{2+q-\theta(1+q)}{(1+q)(1-\theta)}} + \left(\frac{\delta^\theta}{\delta}\right)^{\frac{2+q-\theta(1+q)}{(1+q)(1-\theta)}}\\
&= 2\left(\frac{\delta^\theta}{\delta}\right)^{\frac{2+q-\theta(1+q)}{(1+q)(1-\theta)}}.
\end{align*}

\subsubsection*{\textnormal{\textbf{Case 2:} suppose $0 \leq \theta < \frac{p}{1+q}$.}}
By (\ref{ltapprox}) and (\ref{Lpapprox}), for sufficiently small $\delta > 0$,
$$
L_p^{-p} < L_q^{-p} < l_p^{-p} < l_q^{-p},
$$
with the gaps between the four integers $L_p, L_q, l_p$ and $l_q$  arbitrarily large. Then, for $k = l_p + 1, \dots, L_q$, we have
$$
S_{p, q}^k \subset B(0, \delta^\theta),
$$
while the turns in this region are separated by at least $\delta$ on the horizontal and vertical axes. Therefore they should be covered individually by at least
$$
\frac{\mathcal{H}^1(S_{p, q}^k)}{\delta} \approx_p \frac{k^{-p}}{\delta}
$$
squares of sidelength $\delta$. 

Hence
\begin{align}\label{case3lb}
N_\delta(S_{p, q} \cap B(0, \delta^\theta)) \gtrsim_p \sum\limits_{k=l_p}^{L_q} \frac{k^{-p}}{\delta}.
\end{align}
This sum may be estimated using Lemma \ref{circumlem}. If $p < 1$, then
\begin{align*}
N_\delta(S_{p, q} \cap B(0, \delta^\theta)) &\gtrsim_p \frac{L_q^{1-p} - l_p^{1-p}}{\delta} \\
&\approx_{p,q} \delta^{\frac{p-1}{1+q} - 1} \\
&= \left(\frac{\delta^\theta}{\delta}\right)^{\frac{2+q-p}{(1+q)(1-\theta)}}.
\end{align*}
On the other hand, if $p = 1$, then
\begin{align}\label{estlb}
N_\delta(S_{p, q} \cap B(0, \delta^\theta)) &\gtrsim_p \frac{\log(L_q) - \log(l_p)}{\delta} \nonumber\\
&\approx_{p, q} \delta^{- 1} |\log(\delta)| \nonumber\\
&\geq \left(\frac{\delta^\theta}{\delta}\right)^{\frac{1}{(1-\theta)}}.
\end{align}
Finally, if $p > 1$, then
\begin{align*}
N_\delta(S_{p, q} \cap B(0, \delta^\theta)) &\gtrsim_p \frac{l_p ^{1-p} - L_q^{1-p}}{\delta} \\
&\approx_{p,q} \delta^{\frac{(p-1)\theta}{p} - 1} \\
&= \left(\frac{\delta^\theta}{\delta}\right)^{\frac{p - \theta(p-1)}{p(1-\theta)}}.
\end{align*}
In each case we obtain the desired lower bound.

For the upper bound, we consider a cover of three parts. First, cover turns indexed by $k \geq L_q$ by covering the rectangle

$$
[-(2\pi L_q)^{-p}, (2\pi L_q)^{-p}] \times[-(2\pi L_q)^{-q}, (2\pi L_q)^{-q}] 
$$
by 
$$
\approx_{p,q} \frac{L_q^{-p}L_q^{-q}}{\delta^2} 
$$
squares of sidelength $\delta$. The remaining two portions may then be covered as in (\ref{case2lb}) and (\ref{case3lb}). Hence 
\begin{align*}
N_\delta(S_{p, q} \cap B(0, \delta^\theta)) &\lesssim_{p,q} \frac{L_q^{-p}L_q^{-q}}{\delta^2} + \sum\limits_{k=l_p}^{L_q} \frac{k^{-p}}{\delta} + \sum\limits_{k = l_q}^{l_p} \delta^{\theta - 1}.
\end{align*}
We now apply Lemma \ref{circumlem} in each case. If $p < 1$, then
\begin{align*}
N_\delta(S_{p, q} \cap B(0, \delta^\theta)) &\lesssim_{p,q} \delta^{\frac{p}{1+q} + \frac{q}{1+q} -2} + \delta^{-1}(L_q^{1-p} - l_p^{1-p}) + \delta^{\theta -1}(l_p - l_q)\\
&\lesssim_{p,q} \left(\frac{\delta^{\theta}}{\delta}\right)^{\frac{2+q-p}{(1-\theta)(1+q)}}.
\end{align*}
On the other hand, if $p = 1$, then
\begin{align*}
N_\delta(S_{p, q} \cap B(0, \delta^\theta)) &\lesssim_{p,q} \delta^{\frac{p}{1+q} + \frac{q}{1+q} -2} + \delta^{-1}(\log L_q - \log l_p) + \delta^{\theta -1}(l_p - l_q)\\
&\lesssim_{p,q} \left(\frac{\delta^{\theta}}{\delta}\right)^{\frac{1}{1-\theta}}.
\end{align*}
Finally, if $p > 1$, then
\begin{align*}
N_\delta(S_{p, q} \cap B(0, \delta^\theta)) &\lesssim_{p,q} \delta^{\frac{p}{1+q} + \frac{q}{1+q} -2} + \delta^{-1}(l_p^{1-p} - L_q^{1-p}) + \delta^{\theta -1}(l_p - l_q)\\
&\lesssim_{p,q} \left(\frac{\delta^{\theta}}{\delta}\right)^{\frac{p - (p-1)\theta}{(1-\theta)p}},
\end{align*}
which completes the proof.
\end{proof}

\section*{Acknowledgement}
SAB was supported by a \emph{Carnegie Trust PhD Scholarship} (PHD060287) and \emph{LMS Early-Career Fellowship} (ECF-1920-85). KJF and JMF were supported by an \emph{EPSRC Standard Grant} (EP/R015104/1). JMF was also supported by a \emph{Leverhulme Trust Research Project Grant} (RPG-2019-034). The authors would like to thank David Dritschel for helpful discussion on the physical applications of elliptical spirals. 

\bibliographystyle{amsplain}

\end{document}